\documentclass[onefignum,onetabnum]{siamart220329}

\usepackage{amsfonts,amssymb}

\newsiamremark{remark}{Remark}
\newsiamremark{assumption}{Assumption}
\crefname{assumption}{Assumption}{Assumptions}

\headers{Lipschitz Continuity Results for Minimax Solutions}{M. I. Gomoyunov}

\title{Lipschitz Continuity Results for Minimax Solutions of Path-Dependent Hamilton--Jacobi Equations\thanks{Submitted to the editors 23.12.2024.
\funding{This work was supported by the Ministry of Science and Higher Education of the Russian Federation within a state contract (project FEWS-2024-0009).}}}

\author{Mikhail I. Gomoyunov\thanks{N. N. Krasovskii Institute of Mathematics and Mechanics of the Ural Branch of the Russian Academy of Sciences, Ekaterinburg, 620108, Russia;
Ural Federal University, Ekaterinburg, 620002, Russia;
Udmurt State University, Izhevsk, 426034, Russia (\email{m.i.gomoyunov@gmail.com}).}}

\usepackage{amsopn}

\def\Lip{\operatorname{Lip}}
\def\rd{\mathrm{d}}

\ifpdf
\hypersetup{
  pdftitle={Lipschitz Continuity Results for Minimax Solutions of Path-Dependent Hamilton--Jacobi Equations},
  pdfauthor={M. I. Gomoyunov}
}
\fi

\begin{document}

\maketitle

\begin{abstract}
We consider a Cauchy problem for a (first-order) path-dependent Hamilton--Jacobi equation with coinvariant derivatives and a right-end boundary condition.
Such problems arise naturally in the study of properties of the value functional in (deterministic) optimal control problems and differential games for time-delay systems.
We prove that, under certain assumptions on the Hamiltonian and the boundary functional, a minimax (generalized) solution of the Cauchy problem satisfies Lipschitz conditions both in time and functional variables.
In the latter case, Lipschitz conditions are formulated with respect to the uniform norm and some special norm of the space of continuous functions.
\end{abstract}



\section{Introduction}

    The theory of generalized solutions of boundary value problems for various path-dependent Hamilton--Jacobi equations is currently quite well developed.
    In particular, basic results on the existence and uniqueness of a generalized solution have been obtained;
    its continuous dependence on the parameters of the problem has been established;
    several approaches to its approximation have been proposed;
    and applications for studying optimal control problems and differential games for time-delay systems have been given.
    One of the further directions of research within this theory, which has not been covered in the literature so far, is related to the analysis of regularity properties of a generalized solution.
    The present paper belongs to this direction of research and is devoted to the study of Lipschitz continuity properties of a generalized solution with respect to both time and functional variables.

    To be precise, we consider a (first-order) path-dependent Hamilton--Jacobi equation with coinvariant derivatives and a Cauchy problem for this equation under a right-end boundary condition.
    Recall that such Hamilton--Jacobi equations arise, for example, as Bellman and Isaacs equations in (deterministic) optimal control problems and differential games for time-delay systems.
    For the considered Cauchy problem, both the minimax \cite{Subbotin_1995} and viscosity \cite{Crandall_Lions_1983} approaches to the notion of a generalized solution were successfully developed, and their equivalence was established under rather general assumptions.
    The reader is referred to a recent survey paper \cite{Gomoyunov_Lukoyanov_RMS_2024} focused on the detailed exposition of the main elements of the theory under discussion and including a large bibliography.
    We note only that, within the present paper, we prefer to deal with minimax solutions in view of the choice of the technique for proving the main results, while the results themselves can be appropriately reformulated in terms of viscosity solutions, understood in one particular sense or another.

    The paper contains a number of results on Lipschitz continuity properties of the minimax solution of the considered Cauchy problem both in time and functional variables.
    The Lipschitz conditions in the functional variable are formulated in local and global form and with respect to the uniform norm and some special norm of the space of continuous functions.
    Note that similar Lipschitz continuity properties are known for the value functional in optimal control problems and differential games for time-delay systems, and their proofs are relatively straightforward and based on Gronwall inequality.
    However, such arguments cannot be directly applied to the minimax solution in the case of a general Hamiltonian, which is not necessarily of Bellman or Isaacs form, and a more complex technique is required.
    The approach proposed in this paper goes back to the proof of the so-called comparison principle for (lower and upper) minimax solutions of path-dependent Hamilton--Jacobi equations and uses the technique of coinvariantly smooth Lyapunov--Krasovskii functionals \cite[Lemma 7.7]{Lukoyanov_2003_1} (see also \cite[Lemma 1]{Gomoyunov_Lukoyanov_RMS_2024}).
    Moreover, it utilizes some ideas proposed in the theory of viscosity solutions of (usual) Hamilton--Jacobi equations with first-order partial derivatives when Lipschitz continuity properties of such solutions are investigated (see, e.g., \cite[Theorem 8.1]{Barles_2013}).

    As a secondary motivation for the paper, the following two aspects can be mentioned.
    First, Lipschitz continuity properties close to those established in the paper are sometimes included in a definition of a viscosity solution of a Cauchy problem for a path-dependent Hamilton--Jacobi equation (see, e.g., \cite{Kaise_2015,Zhou_2019,Plaksin_2020_JOTA,Zhou_2020_1,Plaksin_2021_SIAM,Kaise_2022,Zhou_2022,Hernandez-Hernandez_Kaise_2024}).
    These additional requirements are mainly imposed in order to obtain the uniqueness of a viscosity solution.
    On the other hand, the existence of such a viscosity solution is proved only in special cases where the Cauchy problem arises from some optimal control problem or differential game for a time-delay system and follows from the fact that the corresponding value functional is a viscosity solution.
    In this regard, taking the existence of a minimax solution into account, the results of the paper may be useful in proving the existence of a viscosity solution in the case of a general Hamiltonian.
    Second, the local Lipschitz continuity property in the functional variable with respect to the special norm allows us to obtain a novel and convenient infinitesimal criterion for a minimax solution in terms of a pair of inequalities for its derivatives in finite-dimensional (or single-valued) directions \cite[Section 5.4]{Lukoyanov_2006_IMM_Eng}.

    The rest of the paper is organized as follows.
    The considered Cauchy problem for the path-dependent Hamilton--Jacobi equation with coinvariant derivatives and the right-end boundary condition is described in \cref{section_1}, where also the notion of a minimax solution of this problem is recalled.
    The Lipschitz continuity results for the minimax solution in the functional variable with respect to the uniform norm and in the time variable are given in \cref{section_main}.
    The Lipschitz continuity results involving the special norm of the space of continuous functions are obtained in \cref{section_last}, where also the infinitesimal criterion for the minimax solution is presented.

\section{Path-dependent Hamilton--Jacobi equation}
\label{section_1}

    Let numbers $n \in \mathbb{N}$, $T > 0$, and $h \geq 0$ be fixed.
    Consider the Banach space $C([- h, T], \mathbb{R}^n)$ of continuous functions $x \colon [- h, T] \to \mathbb{R}^n$ with the uniform norm $\|x(\cdot)\|_\infty \doteq \max_{\tau \in [- h, T]} \|x(\tau)\|$, where $\|\cdot\|$ stands for the Euclidean norm in $\mathbb{R}^n$.
    Given $t \in [0, T]$ and $x(\cdot) \in C([- h, T], \mathbb{R}^n)$, define a function $x(\cdot \wedge t) \in C([- h, T], \mathbb{R}^n)$ by $x(\tau \wedge t) \doteq x(\tau)$ for all $\tau \in [- h, t]$ and $x(\tau \wedge t) \doteq x(t)$ for all $\tau \in (t, T]$.

    A functional $\varphi \colon [0, T] \times C([- h, T], \mathbb{R}^n) \to \mathbb{R}$ is called {\it non-anticipative} if, for any $t \in [0, T)$ and $x(\cdot)$, $y(\cdot) \in C([- h, T], \mathbb{R}^n)$ such that $x(\cdot \wedge t) = y(\cdot \wedge t)$, the equality $\varphi(t, x(\cdot)) = \varphi(t, y(\cdot))$ is valid.

    For every point $(t, x(\cdot)) \in [0, T] \times C([- h, T], \mathbb{R}^n)$, consider the set $\Lip(t, x(\cdot))$ of functions $y(\cdot) \in C([- h, T], \mathbb{R}^n)$ that satisfy the equality $y(\cdot \wedge t) = x(\cdot \wedge t)$ and are Lipschitz continuous on $[t, T]$.
    According to \cite{Kim_1999,Lukoyanov_2003_1} (see also, e.g., \cite[Section 3.1]{Gomoyunov_Lukoyanov_RMS_2024}), a functional $\varphi \colon [0, T] \times C([- h, T], \mathbb{R}^n) \to \mathbb{R}$ is called {\it coinvariantly differentiable} ($ci$-{\it differentiable}) at a point $(t, x(\cdot)) \in [0, T) \times C([- h, T], \mathbb{R}^n)$ if there exist $\partial_t \varphi(t, x(\cdot)) \in \mathbb{R}$ and $\nabla \varphi(t, x(\cdot)) \in \mathbb{R}^n$ such that, for any $y(\cdot) \in \Lip(t, x(\cdot))$,
    \begin{displaymath}
        \frac{1}{\tau - t}
        \biggl(\varphi(\tau, y(\cdot)) - \varphi(t, x(\cdot))
        - \partial_t \varphi(t, x(\cdot)) (\tau - t)
        - \langle \nabla \varphi(t, x(\cdot)), y(\tau) - x(t) \rangle \biggr)
        \to 0
    \end{displaymath}
    as $\tau \to t^+$.
    Here, $\langle \cdot, \cdot \rangle$ stands for the inner product in $\mathbb{R}^n$.
    In this case, $\partial_t \varphi(t, x(\cdot))$ and $\nabla \varphi(t, x(\cdot))$ are called {\it $ci$-derivatives} of $\varphi$ at $(t, x(\cdot))$.
    Moreover, $\varphi$ is called {\it $ci$-smooth} if it is continuous, $ci$-differentiable at every point $(t, x(\cdot)) \in [0, T) \times C([- h, T], \mathbb{R}^n)$, and the $ci$-derivatives $\partial_t \varphi \colon [0, T) \times C([- h, T], \mathbb{R}^n) \to \mathbb{R}$, $\nabla \varphi \colon [0, T) \times C([- h, T], \mathbb{R}^n) \to \mathbb{R}^n$ are continuous.
    Note that, in this case, the mappings $\varphi$, $\partial_t \varphi$, and $\nabla \varphi$ are automatically non-anticipative \cite[p. 256]{Gomoyunov_Lukoyanov_RMS_2024}.

    Suppose that a mapping $H \colon [0, T] \times C([- h, T], \mathbb{R}^n) \times \mathbb{R}^n \to \mathbb{R}$, called the {\it Hamiltonian}, and a functional $\sigma \colon C([- h, T], \mathbb{R}^n) \to \mathbb{R}$, called the {\it boundary functional}, are given and satisfy the basic assumption below.
    \begin{assumption} \label{assumption_basic}
        The following conditions are satisfied.

        \smallskip

        \noindent (i)
        The mapping $H$ is continuous.

        \smallskip

        \noindent (ii)
        There exists $c_H > 0$ such that
        \begin{displaymath}
            |H(t, x(\cdot), s_1) - H(t, x(\cdot), s_2)|
            \leq c_H (1 + \|x(\cdot \wedge t)\|_\infty) \|s_1 - s_2\|
        \end{displaymath}
        for all $(t, x(\cdot)) \in [0, T] \times C([- h, T], \mathbb{R}^n)$ and $s_1$, $s_2 \in \mathbb{R}^n$.

        \smallskip

        \noindent (iii)
        For any compact set $D \subset C([- h, T], \mathbb{R}^n)$, there exists $\lambda_H \doteq \lambda_H(D) > 0$ such that
        \begin{equation} \label{H_lip}
            |H(t, x_1(\cdot), s) - H(t, x_2(\cdot), s)|
            \leq \lambda_H (1 + \|s\|)
            \|x_1(\cdot \wedge t) - x_2(\cdot \wedge t)\|_\infty
        \end{equation}
        for all $t \in [0, T]$, $x_1(\cdot)$, $x_2(\cdot) \in D$, and $s \in \mathbb{R}^n$.

        \smallskip

        \noindent (iv)
        The functional $\sigma$ is continuous.
    \end{assumption}

    Observe that condition (iii) implies that, for every fixed $s \in \mathbb{R}^n$, the functional
    $[0, T] \times C([-h, T], \mathbb{R}^n) \ni (t, x(\cdot)) \mapsto H(t, x(\cdot), s)$ is non-anticipative.

    Consider the {\it path-dependent Hamilton--Jacobi equation} with $ci$-derivatives
    \begin{subequations} \label{Cauchy_problem}
    \begin{equation} \label{HJ}
        \partial_t \varphi(t, x(\cdot)) + H \bigl( t, x(\cdot), \nabla \varphi(t, x(\cdot)) \bigr)
        = 0
        \quad \forall (t, x(\cdot)) \in [0, T) \times C([- h, T], \mathbb{R}^n)
    \end{equation}
    and the {\it Cauchy problem} for this equation and the right-end {\it boundary condition}
    \begin{equation} \label{boundary_condition}
        \varphi(T, x(\cdot))
        = \sigma(x(\cdot))
        \quad \forall x(\cdot) \in C([- h, T], \mathbb{R}^n).
    \end{equation}
    \end{subequations}
    The unknown here is a non-anticipative functional $\varphi \colon [0, T] \times C([- h, T], \mathbb{R}^n) \to \mathbb{R}$.

    For every point $(t, x(\cdot)) \in [0, T] \times C([- h, T], \mathbb{R}^n)$, consider the set $Y(t, x(\cdot))$ of functions $y(\cdot) \in \Lip(t, x(\cdot))$ that satisfy the inequality $\|\dot{y}(\tau)\| \leq c_H (1 + \|y(\cdot \wedge \tau)\|_\infty)$ for almost every (a.e.) $\tau \in [t, T]$.
    Here, the number $c_H$ is taken from \cref{assumption_basic}, (ii), and $\dot{y}(\tau) \doteq \rd y(\tau) / \rd \tau$.
    Note that the set $Y(t, x(\cdot))$ is compact in $C([- h, T], \mathbb{R}^n)$ (see, e.g., \cite[Proposition 4.1]{Lukoyanov_2003_1}).

    In accordance with \cite[Definition 5.1]{Lukoyanov_2003_1} (see also, e.g., \cite[Section 3.3]{Gomoyunov_Lukoyanov_RMS_2024}), a functional $\varphi \colon [0, T] \times C([- h, T], \mathbb{R}^n) \to \mathbb{R}$ is called a {\it minimax solution} of the Cauchy problem \cref{Cauchy_problem} if it is non-anticipative and continuous, satisfies the boundary condition \cref{boundary_condition}, and has the following property:
    for any $(t, x(\cdot)) \in [0, T] \times C([-h, T], \mathbb{R}^n)$ and $s \in \mathbb{R}^n$, there exists $y(\cdot) \in Y(t, x(\cdot))$ such that
    \begin{displaymath}
        \varphi(\tau, y(\cdot))
        = \varphi(t, x(\cdot))
        + \int_{t}^{\tau} \bigl( \langle s, \dot{y}(\xi) \rangle - H(\xi, y(\cdot), s) \bigr) \, \rd \xi
        \quad \forall \tau \in [t, T].
    \end{displaymath}
    Under \cref{assumption_basic}, a minimax solution $\varphi$ exists and is unique \cite[Theorem 1]{Gomoyunov_Lukoyanov_Plaksin_2021}.

    The goal of the present paper is to study Lipschitz continuity properties of $\varphi$ under some additional assumptions on the Hamiltonian $H$ and/or the boundary functional $\sigma$.
    In the next section, the results on local and global Lipschitz continuity properties of $\varphi$ in the (second) functional variable $x(\cdot)$ with respect to the uniform norm $\|\cdot\|_\infty$ are given.
    Moreover, a result on Lipschitz continuity property of $\varphi$ in the (first) time variable $t$ is also obtained.

\section{Uniform norm Lipschitz continuity results}
\label{section_main}

    Consider the following additional assumption on the boundary functional $\sigma$ from \cref{boundary_condition}.
    \begin{assumption} \label{assumption_sigma_lip}
        For any compact set $D \subset C([- h, T], \mathbb{R}^n)$, there exists a number $\lambda_\sigma \doteq \lambda_\sigma(D) > 0$ such that, for any $x_1(\cdot)$, $x_2(\cdot) \in D$,
        \begin{equation} \label{sigma_lip}
            |\sigma(x_1(\cdot)) - \sigma(x_2(\cdot))|
            \leq \lambda_\sigma
            \|x_1(\cdot) - x_2(\cdot)\|_\infty.
        \end{equation}
    \end{assumption}

    Note that \cref{assumption_sigma_lip} implies condition (iv) from \cref{assumption_basic}.

    \begin{theorem} \label{theorem_main}
        Let conditions {\rm (i)}--{\rm (iii)} from \cref{assumption_basic} and \cref{assumption_sigma_lip} hold.
        Then, the minimax solution $\varphi$ of the Cauchy problem \cref{Cauchy_problem} satisfies the following local Lipschitz condition with respect to the functional variable $x(\cdot)$: for any compact set $D \subset C([- h, T], \mathbb{R}^n)$, there exists $\lambda_\varphi \doteq \lambda_\varphi(D) > 0$ such that
        \begin{equation} \label{varphi_lip}
            |\varphi(t, x_1(\cdot)) - \varphi(t, x_2(\cdot))|
            \leq \lambda_\varphi
            \|x_1(\cdot \wedge t) - x_2(\cdot \wedge t)\|_\infty
        \end{equation}
        for all $t \in [0, T]$ and $x_1(\cdot)$, $x_2(\cdot) \in D$.
    \end{theorem}

    In general, the proof of \cref{theorem_main} follows the scheme of the proof of the so-called comparison principle for (lower and upper) minimax solutions of path-dependent Hamilton--Jacobi equations and uses the technique of $ci$-smooth Lyapunov--Krasovskii functionals \cite[Lemma 7.7]{Lukoyanov_2003_1} (see also \cite[Lemma 1]{Gomoyunov_Lukoyanov_RMS_2024}).
    In addition, it utilizes some arguments proposed in the theory of viscosity solutions of (usual) Hamilton--Jacobi equations with first-order partial derivatives when Lipschitz continuity properties of such solutions are investigated (see, e.g., \cite[Theorem 8.1]{Barles_2013}).

    \begin{proof}[Proof of \cref{theorem_main}]
    We split the proof into several steps.

    1.
        Let a compact set $D \subset C([- h, T], \mathbb{R}^n)$ be fixed.
        Consider the set $K \doteq K(D)$ consisting of functions $y(\cdot) \in C([- h, T], \mathbb{R}^n)$ for each of which there exist $t \in [0, T]$ and $x(\cdot) \in D$ such that $y(\cdot) \in Y(t, x(\cdot))$.
        Based on, e.g., \cite[Proposition 4.2]{Lukoyanov_2003_1}, one can show that $K$ is compact.
        Choose $\lambda_H \doteq \lambda_H(K)$ and $\lambda_\sigma \doteq \lambda_\sigma(K)$ by \cref{assumption_basic}, (iii), and \cref{assumption_sigma_lip}, respectively.
        Following \cite{Zhou_2020_1}, consider the functional
        \begin{displaymath}
            \gamma (t, x(\cdot))
            \doteq \begin{cases}
                \displaystyle
                \frac{(\|x(\cdot \wedge t)\|_\infty^2 - \|x(t)\|^2)^2}{\|x(\cdot \wedge t)\|_\infty^2} + \|x(t)\|^2
                & \mbox{if } \|x(\cdot \wedge t)\|_\infty > 0, \\
                0 & \mbox{if } \|x(\cdot \wedge t)\|_\infty = 0
              \end{cases}
        \end{displaymath}
        for all $(t, x(\cdot)) \in [0, T] \times C([- h, T], \mathbb{R}^n)$.
        According to \cite[Lemma 2.3]{Zhou_2020_1}, we have
        \begin{equation} \label{gamma_bounds}
            \varkappa \|x(\cdot \wedge t)\|_\infty^2
            \leq \gamma (t, x(\cdot))
            \leq 2 \|x(\cdot \wedge t)\|_\infty^2
        \end{equation}
        for all $(t, x(\cdot)) \in [0, T] \times C([- h, T], \mathbb{R}^n)$ with $\varkappa \doteq (3 - \sqrt{5}) / 2$.
        Moreover (see also \cite[Section 4.1]{Gomoyunov_Lukoyanov_Plaksin_2021}), the functional $\gamma$ is $ci$-smooth and its $ci$-derivatives satisfy the relations
        \begin{equation} \label{nabla_gamma_bound}
            \partial_t \gamma(t, x(\cdot))
            = 0,
            \quad \|\nabla \gamma(t, x(\cdot))\|
            \leq 2 \|x(t)\|
        \end{equation}
        for all $(t, x(\cdot)) \in [0, T) \times C([- h, T], \mathbb{R}^n)$.
        Further, consider the function
        \begin{displaymath}
           a(t)
           \doteq \sqrt{\varkappa} (e^{\lambda_H  (T - t) / \varkappa} - 1)
           + \lambda_\sigma e^{\lambda_H  (T - t) / \varkappa} / \sqrt{\varkappa}
           \quad \forall t \in \mathbb{R}.
        \end{displaymath}
        Note that $a(t) > 0$ for all $t \in [0, T]$, $a(T) = \lambda_\sigma / \sqrt{\varkappa}$, and
        \begin{equation} \label{C_differential_equation}
            \dot{a}(t) + \lambda_H a(t) / \varkappa + \lambda_H / \sqrt{\varkappa}
            = 0
            \quad \forall t \in \mathbb{R}.
        \end{equation}
        Put $\lambda_\varphi \doteq \lambda_\varphi(D) \doteq \sqrt{2} a(0)$.
        Due to \cref{gamma_bounds} and since $a(t) \leq a(0)$ for all $t \in [0, T]$, we conclude that, in order to complete the proof, it remains to show that
        \begin{equation} \label{basic}
            \varphi(t, x_1(\cdot)) - \varphi(t, x_2(\cdot))
            \leq a(t) \sqrt{\gamma(t, x_1(\cdot) - x_2(\cdot))}
        \end{equation}
        for all $t \in [0, T]$ and $x_1(\cdot)$, $x_2(\cdot) \in D$.

    2.
        Fix $\varepsilon > 0$ and consider the functional
        \begin{equation} \label{nu}
            \nu(t, x(\cdot))
            \doteq a(t) \sqrt{\varepsilon + \gamma(t, x(\cdot))}
        \end{equation}
        for all $(t, x(\cdot)) \in [0, T] \times C([- h, T], \mathbb{R}^n)$.
        Note that the functional $\nu$ is $ci$-smooth and, in accordance with \cref{nabla_gamma_bound},
        \begin{equation} \label{nu_derivatives}
            \partial_t \nu(t, x(\cdot))
            = \dot{a}(t) \sqrt{\varepsilon + \gamma(t, x(\cdot))},
            \quad \nabla \nu(t, x(\cdot))
            = \frac{a(t) \nabla \gamma(t, x(\cdot))}{2 \sqrt{\varepsilon + \gamma(t, x(\cdot))}}
        \end{equation}
        for all $(t, x(\cdot)) \in [0, T) \times C([- h, T], \mathbb{R}^n)$.
        In addition, using \cref{gamma_bounds,nabla_gamma_bound}, we derive
        \begin{equation} \label{nu_derivatives_estimate}
            \|\nabla \nu(t, x(\cdot))\|
            \leq a(t) / \sqrt{\varkappa}
            \quad \forall (t, x(\cdot)) \in [0, T) \times C([- h, T], \mathbb{R}^n).
        \end{equation}

        Consider the multivalued mapping
        \begin{displaymath}
            [0, T) \times C([- h, T], \mathbb{R}^n) \times C([- h, T], \mathbb{R}^n) \ni (\tau, y_1(\cdot), y_2(\cdot)) \mapsto F(\tau, y_1(\cdot), y_2(\cdot))
        \end{displaymath}
        such that the set $F(\tau, y_1(\cdot), y_2(\cdot))$ consists of triples $(f_1, f_2, \chi) \in \mathbb{R}^n \times \mathbb{R}^n \times \mathbb{R}$ with
        \begin{equation} \label{F_1}
            \|f_1\|
            \leq c_H (1 + \|y_1(\cdot \wedge \tau)\|_\infty),
            \quad \|f_2\|
            \leq c_H (1 + \|y_2(\cdot \wedge \tau)\|_\infty),
        \end{equation}
        where $c_H$ is taken from \cref{assumption_basic}, (ii), and
        \begin{equation} \label{F_2}
            \begin{aligned}
                & \bigl| \chi - \langle \nabla \nu(\tau, y_1(\cdot) - y_2(\cdot)), f_2 - f_1 \rangle \\
                & \quad - H \bigl( \tau, y_1(\cdot), \nabla \nu(\tau, y_1(\cdot) - y_2(\cdot)) \bigr)
                + H \bigl( \tau, y_2(\cdot), \nabla \nu(\tau, y_1(\cdot) - y_2(\cdot)) \bigr)
                \bigr|
                \leq \varepsilon.
            \end{aligned}
        \end{equation}
        Observe that the mapping $F$ has nonempty compact convex values, is non-anticipative and Hausdorff continuous since $H$ and $\nabla \nu$ are non-anticipative and continuous.

        Now, let $t \in [0, T]$ and $x_1(\cdot)$, $x_2(\cdot) \in D$ be fixed.
        Denote by $S$ the set of triples $(y_1(\cdot), y_2(\cdot), z(\cdot)) \in C([- h, T], \mathbb{R}^n) \times C([- h, T], \mathbb{R}^n) \times C([- h, T], \mathbb{R})$ such that $y_1(\cdot)$, $y_2(\cdot)$, and $z(\cdot)$ satisfy a Lipschitz condition on $[t, T]$, the differential inclusion
        \begin{equation} \label{F_differential_inclusion}
            (\dot{y}_1(\tau), \dot{y}_2(\tau), \dot{z}(\tau))
            \in F(\tau, y_1(\cdot), y_2(\cdot))
        \end{equation}
        for a.e. $\tau \in [t, T]$, and the initial conditions
        \begin{equation} \label{F_initial_condition}
            y_1(\tau)
            = x_1(\tau),
            \quad y_2(\tau)
            = x_2(\tau),
            \quad z(\tau)
            = \varphi(t, x_2(\cdot)) - \varphi(t, x_1(\cdot))
        \end{equation}
        for all $\tau \in [- h, t]$.
        Owing to the properties of $F$ described above and the growth conditions \cref{F_1}, the set $S$ is nonempty and compact in $C([- h, T], \mathbb{R}^n) \times C([- h, T], \mathbb{R}^n) \times C([- h, T], \mathbb{R})$ (in this connection, see, e.g., \cite[Proposition 4.1]{Lukoyanov_2003_1}).

    3.
        Let us prove that there exists a triple $(y_1^\circ(\cdot), y_2^\circ(\cdot), z^\circ(\cdot)) \in S$ such that
        \begin{equation} \label{z^circ}
            z^\circ(T)
            \geq \varphi(T, y_2^\circ(\cdot)) - \varphi(T, y_1^\circ(\cdot)).
        \end{equation}
        For every $\tau \in [t, T]$, consider the set
        \begin{displaymath}
            M(\tau)
            \doteq \bigl\{ (y_1(\cdot), y_2(\cdot), z(\cdot)) \in S \colon
            z(\tau)
            \geq \varphi(\tau, y_2(\cdot)) - \varphi(\tau, y_1(\cdot)) \bigr\}.
        \end{displaymath}
        Note that $M(t) \neq \emptyset$ owing to \cref{F_initial_condition} and because $\varphi$ is non-anticipative.
        Put
        \begin{displaymath}
            \vartheta
            \doteq \max \bigl\{ \tau \in [t, T] \colon
            M(\tau) \neq \emptyset \bigr\},
        \end{displaymath}
        where the maximum is attained since $S$ is compact and $\varphi$ is continuous.
        So, we need to show that $\vartheta = T$.
        Arguing by contradiction, suppose that $\vartheta < T$.
        Choose arbitrarily
        \begin{equation} \label{proof_lemma_y_1_y_2_z}
            (y_1(\cdot), y_2(\cdot), z(\cdot)) \in M(\vartheta).
        \end{equation}

        Take $s \doteq \nabla \nu(\vartheta, y_1(\cdot) - y_2(\cdot))$ and, using the fact that $\varphi$ is the minimax solution, find $y_1^\ast(\cdot) \in Y(\vartheta, y_1(\cdot))$ and $y_2^\ast(\cdot) \in Y(\vartheta, y_2(\cdot))$ such that
        \begin{equation} \label{proof_lemma_y^+}
            \varphi(\tau, y_2^\ast(\cdot)) - \varphi(\tau, y_1^\ast(\cdot))
            = \varphi(\vartheta, y_2(\cdot)) - \varphi(\vartheta, y_1(\cdot))
            + z^\ast(\tau)
            \quad \forall \tau \in [\vartheta, T],
        \end{equation}
        where we denote
        \begin{displaymath}
            z^\ast(\tau)
            \doteq \int_{\vartheta}^{\tau} \bigl( \langle s, \dot{y}_2^\ast(\xi) - \dot{y}_1^\ast(\xi) \rangle + H(\xi, y_1^\ast(\cdot), s) - H(\xi, y_2^\ast(\cdot), s) \bigr) \, \rd \xi.
        \end{displaymath}
        Then, due to continuity of $H$ and $\nabla \nu$, there exists $\delta \in (0, T - \vartheta)$ such that
        \begin{displaymath}
            \begin{aligned}
                & \bigl| \dot{z}^\ast(\tau) - \langle \nabla \nu(\tau, y_1^\ast(\cdot) - y_2^\ast(\cdot)), \dot{y}_2^\ast(\tau) - \dot{y}_1^\ast(\tau) \rangle \\
                & \quad - H \bigl( \tau, y_1^\ast(\cdot), \nabla \nu(\tau, y_1^\ast(\cdot) - y_2^\ast(\cdot)) \bigr)
                + H \bigl( \tau, y_2^\ast(\cdot), \nabla \nu(\tau, y_1^\ast(\cdot) - y_2^\ast(\cdot)) \bigr) \bigr|
                \leq \varepsilon
            \end{aligned}
        \end{displaymath}
        for a.e. $\tau \in [\vartheta, \vartheta + \delta]$.
        Hence, in accordance with \cref{F_1,F_2}, we conclude that
        \begin{equation} \label{proof_lemma_inclusion}
            (\dot{y}_1^\ast(\tau), \dot{y}_2^\ast(\tau), \dot{z}^\ast(\tau))
            \in F(\tau, y_1^\ast(\cdot), y_2^\ast(\cdot))
        \end{equation}
        for a.e. $\tau \in [\vartheta, \vartheta + \delta]$.

        Consider the functions $\bar{y}^\ast_1 (\cdot) \doteq y_1^\ast(\cdot \wedge \vartheta + \delta)$ and $\bar{y}^\ast_2(\cdot) \doteq y_2^\ast(\cdot \wedge \vartheta + \delta)$, and define $\bar{z}^\ast(\tau) \doteq z(\tau)$ if $\tau \in [- h, \vartheta]$, $\bar{z}^\ast(\tau) \doteq z(\vartheta) + z^\ast(\tau)$ if $\tau \in (\vartheta, \vartheta + \delta]$, and
        \begin{displaymath}
            \begin{aligned}
                 & \bar{z}^\ast(\tau)
                 \doteq z(\vartheta) + z^\ast(\vartheta + \delta) \\
                 & \quad + \int_{\vartheta + \delta}^{\tau} \Bigl( H \bigl( \xi, \bar{y}^\ast_1(\xi), \nabla \nu(\xi, \bar{y}^\ast_1(\cdot) - \bar{y}^\ast_2(\cdot)) \bigr) - H \bigl( \xi, \bar{y}^\ast_2(\xi), \nabla \nu(\xi, \bar{y}^\ast_1(\cdot) - \bar{y}^\ast_2(\cdot)) \bigr) \Bigr) \, \rd \xi
            \end{aligned}
        \end{displaymath}
        if $\tau \in (\vartheta + \delta, T]$.
        In view of \cref{proof_lemma_y_1_y_2_z,proof_lemma_inclusion}, we obtain $(\bar{y}^\ast_1(\cdot), \bar{y}^\ast_2(\cdot), \bar{z}^\ast(\cdot)) \in S$.
        Moreover, from \cref{proof_lemma_y_1_y_2_z,proof_lemma_y^+}, and recalling that $\varphi$ is non-anticipative, we derive
        \begin{displaymath}
            \begin{aligned}
                \bar{z}^\ast(\vartheta + \delta)
                & \geq \varphi(\vartheta, y_2(\cdot)) - \varphi(\vartheta, y_1(\cdot)) + z^\ast(\vartheta + \delta) \\
                & = \varphi(\vartheta + \delta, y_2^\ast(\cdot)) - \varphi(\vartheta + \delta, y_1^\ast(\cdot)) \\
                & = \varphi(\vartheta + \delta, \bar{y}^\ast_2(\cdot)) - \varphi(\vartheta + \delta, \bar{y}^\ast_1(\cdot)).
            \end{aligned}
        \end{displaymath}
        Thus, the inclusion $(\bar{y}^\ast_1(\cdot), \bar{y}^\ast_2(\cdot), \bar{z}^\ast(\cdot)) \in M(\vartheta + \delta)$ takes place, which contradicts the definition of $\vartheta$.

    4.
        By the functions $(y_1^\circ(\cdot), y_2^\circ(\cdot), z^\circ(\cdot))$ from step 3, define the function
        \begin{equation} \label{mu}
            \mu(\tau)
            \doteq \nu(\tau, \Delta y^\circ(\cdot))
            + z^\circ(\tau) - \varepsilon (\tau - t)
            \quad \forall \tau \in [t, T],
        \end{equation}
        where $\Delta y^\circ(\cdot) \doteq y_1^\circ(\cdot) - y_2^\circ(\cdot)$.
        Based on the fact that $\nu$ is $ci$-smooth and then applying \cite[Proposition 1]{Gomoyunov_Lukoyanov_RMS_2024}, we find that $\mu(\cdot)$ is continuous, satisfies a Lipschitz condition on $[t, \vartheta]$ for every $\vartheta \in (t, T)$, and, for a.e. $\tau \in [t, T]$,
        \begin{displaymath}
            \dot{\mu}(\tau)
            = \partial_t \nu(\tau, \Delta y^\circ(\cdot))
            + \langle \nabla \nu(\tau, \Delta y^\circ(\cdot)), \dot{y}_1^\circ(\tau) - \dot{y}_2^\circ(\tau) \rangle + \dot{z}^\circ(\tau) - \varepsilon.
        \end{displaymath}
        Consequently, owing to the inclusion $(y_1^\circ(\cdot), y_2^\circ(\cdot), z^\circ(\cdot)) \in S$ and \cref{F_2}, we derive
        \begin{equation} \label{dot_mu_estimate}
            \begin{aligned}
                \dot{\mu}(\tau)
                & \leq \partial_t \nu(\tau, \Delta y^\circ(\cdot)) \\
                & \quad + H \bigl( \tau, y_1^\circ(\cdot), \nabla \nu(\tau, \Delta y^\circ(\cdot)) \bigr)
                - H \bigl( \tau, y_2^\circ(\cdot), \nabla \nu(\tau, \Delta y^\circ(\cdot)) \bigr)
            \end{aligned}
        \end{equation}
        for a.e. $\tau \in [t, T]$.
        So, by \cref{gamma_bounds,nu_derivatives,nu_derivatives_estimate,F_1} and using the choice of $\lambda_H$, we get
        \begin{displaymath}
            \begin{aligned}
                \dot{\mu}(\tau)
                & \leq \partial_t \nu(\tau, \Delta y^\circ(\cdot))
                + \lambda_H (1 + \|\nabla \nu(\tau, \Delta y^\circ(\cdot))\|)
                \|\Delta y^\circ(\cdot \wedge \tau)\|_\infty \\
                & \leq (\dot{a}(\tau) + \lambda_H (1 /  \sqrt{\varkappa} + a(\tau) / \varkappa))
                \sqrt{\varepsilon + \gamma(\tau, \Delta y^\circ(\cdot))}
            \end{aligned}
        \end{displaymath}
        for a.e. $\tau \in [t, T]$.
        Since $a(\cdot)$ satisfies the differential equation \cref{C_differential_equation}, we conclude that $\dot{\mu}(\tau) \leq 0$ for a.e. $\tau \in [t, T]$ and, therefore, $\mu(t) \geq \mu(T)$.

        Due to \cref{nu,F_initial_condition} and the fact that $\gamma$ is non-anticipative, we have
        \begin{equation} \label{mu_t}
            \mu(t)
            = a(t) \sqrt{\varepsilon + \gamma(t, x_1(\cdot) - x_2(\cdot))}
            + \varphi(t, x_2(\cdot)) - \varphi(t, x_1(\cdot)).
        \end{equation}
        On the other hand, according to \cref{boundary_condition,gamma_bounds,z^circ,F_1} and recalling the choice of $\lambda_\sigma$ and the equality $a(T) = \lambda_\sigma / \sqrt{\varkappa}$, we obtain
        \begin{displaymath}
            \begin{aligned}
                \mu(T)
                & \geq a(T) \sqrt{\varepsilon + \gamma(T, \Delta y^\circ(\cdot))}
                + \sigma(y_2^\circ(\cdot)) - \sigma(y_1^\circ(\cdot)) - \varepsilon (T - t) \\
                & \geq \lambda_\sigma \sqrt{\gamma(T, \Delta y^\circ(\cdot))} / \sqrt{\varkappa}
                - \lambda_\sigma \|\Delta y^\circ(\cdot)\|_\infty - \varepsilon (T - t) \\
                & \geq - \varepsilon (T - t).
            \end{aligned}
        \end{displaymath}
        As a result, we come to the inequality
        \begin{equation} \label{last_inequality}
            a(t) \sqrt{\varepsilon + \gamma(t, x_1(\cdot) - x_2(\cdot))}
            + \varphi(t, x_2(\cdot)) - \varphi(t, x_1(\cdot))
            \geq - \varepsilon (T - t),
        \end{equation}
        which yields \cref{basic} if we pass to the limit as $\varepsilon \to 0^+$.
    \end{proof}

    Analyzing the proof of \cref{theorem_main}, we conclude that, if Lipschitz conditions from \cref{assumption_basic}, (iii), and \cref{assumption_sigma_lip} are global (i.e., $\lambda_H(D)$ and $\lambda_\sigma(D)$ do not actually depend on the choice of a compact set $D \subset C([- h, T], \mathbb{R}^n)$), then the minimax solution $\varphi$ of the Cauchy problem \cref{Cauchy_problem} also satisfies the corresponding global Lipschitz condition.
    This result is formulated below.

    \begin{assumption} \label{assumption_sigma_lip_glob}
        The following conditions are satisfied.

        \smallskip

        \noindent (i)
            There exists a number $\lambda_H > 0$ such that inequality \cref{H_lip} is valid for all $t \in [0, T]$, $x_1(\cdot)$, $x_2(\cdot) \in C([- h, T], \mathbb{R}^n)$, and $s \in \mathbb{R}^n$.

        \smallskip

        \noindent (ii)
            There exists a number $\lambda_\sigma > 0$ such that inequality \cref{sigma_lip} is valid for all $x_1(\cdot)$, $x_2(\cdot) \in C([- h, T], \mathbb{R}^n)$.
    \end{assumption}

    \begin{corollary} \label{corollary_main_glob}
        Let conditions {\rm (i)} and {\rm (ii)} from \cref{assumption_basic} and \cref{assumption_sigma_lip_glob} hold.
        Then, the minimax solution $\varphi$ of the Cauchy problem \cref{Cauchy_problem} satisfies the following condition: there exists a number $\lambda_\varphi > 0$ such that inequality \cref{varphi_lip} is valid for all $t \in [0, T]$ and $x_1(\cdot)$, $x_2(\cdot) \in C([- h, T], \mathbb{R}^n)$.
    \end{corollary}

    In addition, from \cref{theorem_main}, we derive a Lipschitz continuity property of the minimax solution $\varphi$ with respect to the time variable $t$.
    Given a number $\lambda_x > 0$, denote by $D_{\lambda_x}$ the set of functions $x(\cdot) \in C([- h, T], \mathbb{R}^n)$ satisfying the conditions
    \begin{displaymath}
        \|x(\tau_1) - x(\tau_2)\|
        \leq \lambda_x |\tau_1 - \tau_2|
        \quad \forall \tau_1, \tau_2 \in [- h, T],
        \quad \|x(\cdot)\|_\infty
        \leq \lambda_x.
    \end{displaymath}
    Observe that the set $D_{\lambda_x}$ is compact in $C([- h, T], \mathbb{R}^n)$.
    \begin{corollary} \label{corollary_lip_t}
        Under \cref{assumption_basic}, {\rm(i)}--{\rm(iii)}, and \cref{assumption_sigma_lip}, the minimax solution $\varphi$ of the Cauchy problem \cref{Cauchy_problem} satisfies the following Lipschitz condition with respect to the time variable $t$: for any $\lambda_x > 0$, there exists $\lambda_\varphi^\circ \doteq \lambda_\varphi^\circ(\lambda_x) > 0$ such that, for any $t_1$, $t_2 \in [0, T]$ and $x(\cdot) \in D_{\lambda_x}$,
        \begin{equation} \label{corollary_Lip_t_condition}
            |\varphi(t_1, x(\cdot)) - \varphi(t_2, x(\cdot))|
            \leq \lambda_\varphi^\circ |t_1 - t_2|.
        \end{equation}
    \end{corollary}
    \begin{proof}
        Let $\lambda_x > 0$ be fixed.
        By the set $D_{\lambda_x}$, define the set $K \doteq K(D_{\lambda_x})$ as in step 1 of the proof of \cref{theorem_main}.
        Since $K$ is compact in $C([- h, T], \mathbb{R}^n)$, choose $\lambda_\varphi \doteq \lambda_\varphi(K)$ according to \cref{theorem_main}, consider $R_1 > 0$ such that $\|y(\cdot)\|_\infty \leq R_1$ for all $y(\cdot) \in K$, and, using continuity of $H$, take $R_2 > 0$ such that $|H(t, y(\cdot), 0)| \leq R_2$ for all $t \in [0, T]$ and $y(\cdot) \in K$.
        Put $\lambda_\varphi^\circ \doteq \lambda_\varphi^\circ(\lambda_x) \doteq R_2 + \lambda_\varphi ( c_H (1 + R_1) + \lambda_x)$, where $c_H$ is the number from \cref{assumption_basic}, (ii).

        Fix $t_1$, $t_2 \in [0, T]$ and $x(\cdot) \in D_{\lambda_x}$ and suppose for definiteness that $t_1 \leq t_2$.
        Since $\varphi$ is the minimax solution, there exists $y(\cdot) \in Y(t_1, x(\cdot))$ such that
        \begin{displaymath}
            \varphi(t_2, y(\cdot))
            = \varphi(t_1, x(\cdot))
            - \int_{t_1}^{t_2} H(\xi, y(\cdot), 0) \, \rd \xi.
        \end{displaymath}
        Then, noting that $x(\cdot)$, $y(\cdot) \in K$ by construction, we derive
        \begin{displaymath}
            \begin{aligned}
                |\varphi(t_1, x(\cdot)) - \varphi(t_2, x(\cdot))|
                & \leq |\varphi(t_1, x(\cdot)) - \varphi(t_2, y(\cdot))|
                + |\varphi(t_2, y(\cdot)) - \varphi(t_2, x(\cdot))| \\
                & \leq R_2 (t_2 - t_1) + \lambda_\varphi \|y(\cdot \wedge t_2) - x(\cdot \wedge t_2)\|_\infty.
            \end{aligned}
        \end{displaymath}
        In view of the inclusions $x(\cdot) \in D_{\lambda_x}$ and $y(\cdot) \in Y(t_1, x(\cdot))$, we obtain
        \begin{displaymath}
            \begin{aligned}
                \|y(\cdot \wedge t_2) - x(\cdot \wedge t_2)\|_\infty
                & = \max_{\tau \in [t_1, t_2]} \|y(\tau) - x(\tau)\| \\
                & \leq \max_{\tau \in [t_1, t_2]} \|y(\tau) - y(t_1)\|
                + \max_{\tau \in [t_1, t_2]} \|x(t_1) - x(\tau)\| \\
                & \leq \int_{t_1}^{t_2} \|\dot{y}(\tau)\| \, \rd \tau
                + \lambda_x (t_2 - t_1) \\
                & \leq (c_H (1 + R_1) + \lambda_x) (t_2 - t_1).
            \end{aligned}
        \end{displaymath}
        As a result, we conclude that \cref{corollary_Lip_t_condition} holds with the specified $\lambda_\varphi^\circ$.
    \end{proof}

    \begin{remark} \label{remark_definitions_Lip}
        Lipschitz continuity properties close to those established in \cref{theorem_main,corollary_main_glob,corollary_lip_t} are sometimes included in a definition of a viscosity solution of a Cauchy problem for a path-dependent Hamilton--Jacobi equation (see, e.g., \cite{Kaise_2015,Zhou_2019,Zhou_2020_1,Kaise_2022,Zhou_2022,Hernandez-Hernandez_Kaise_2024}).
        In the cited papers, this is mainly done in order to establish the uniqueness of a viscosity solution.
        However, the existence of such a viscosity solution is proved only in special cases where the Cauchy problem arises from some optimal control problem or differential game for a time-delay system and follows from the fact that the corresponding value functional is a viscosity solution.
        In this regard, \cref{theorem_main} and \cref{corollary_main_glob,corollary_lip_t} may be useful in proving the existence of a viscosity solution in the case of a general Hamiltonian $H$, which is not necessarily of Bellman or Isaacs form.
    \end{remark}

\section{Special norm Lipschitz continuity results}
\label{section_last}

    In this section, we focus on the case where the Hamiltonian $H$ and the boundary functional $\sigma$ from \cref{Cauchy_problem} satisfy local Lipschitz continuity conditions in the functional variable $x(\cdot)$ that are stronger than \cref{assumption_basic}, (iii), and \cref{assumption_sigma_lip}, respectively.

    \begin{assumption} \label{assumption_lip_special}
        The following conditions are satisfied.

        \smallskip

        \noindent (i)
            There exist numbers $J \in \mathbb{N}$ and $\vartheta_j \in (0, h]$, where $j \in \overline{1, J}$, such that, for any compact set $D \subset C([- h, T], \mathbb{R}^n)$, there exists a number $\lambda_H^\ast \doteq \lambda_H^\ast(D) > 0$ such that
            \begin{equation} \label{H_lip_special}
                \begin{aligned}
                    & |H(t, x_1(\cdot), s) - H(t, x_2(\cdot), s)|
                    \leq \lambda_H^\ast (1 + \|s\|)
                    \biggl( \|x_1(t) - x_2(t)\|
                     \\
                    & \quad + \sum_{j = 1}^J \|x_1(t - \vartheta_j) - x_2(t - \vartheta_j)\| + \sqrt{\int_{- h}^{t} \|x_1(\tau) - x_2(\tau)\|^2 \, \rd \tau} \biggr)
                \end{aligned}
            \end{equation}
            for all $t \in [0, T]$, $x_1(\cdot)$, $x_2(\cdot) \in D$, and $s \in \mathbb{R}^n$.

        \smallskip

        \noindent (ii)
            For any compact set $D \subset C([- h, T], \mathbb{R}^n)$, there exists $\lambda_\sigma^\ast \doteq \lambda_\sigma^\ast(D) > 0$ such that
            \begin{equation} \label{sigma_lip_special}
                |\sigma(x_1(\cdot)) - \sigma(x_2(\cdot))|
                \leq \lambda_\sigma^\ast \biggl( \|x_1(T) - x_2(T)\|
                + \sqrt{\int_{- h}^{T} \|x_1(\tau) - x_2(\tau)\|^2 \, \rd \tau} \biggr)
            \end{equation}
            for all $x_1(\cdot)$, $x_2(\cdot) \in D$.
    \end{assumption}

    Note that conditions similar to (i) and (ii) are considered, e.g., in \cite{Lukoyanov_2010_IMM_Eng_1,Lukoyanov_2010_IMM_Eng_2}.
    In view of applications to optimal control problems and differential games for time-delay systems, condition (i) makes it possible to deal with particularly important cases of distributed and constant concentrated delays.

    \begin{theorem} \label{theorem_main_2}
        Let conditions {\rm (i)}, {\rm (ii)} from \cref{assumption_basic} and \cref{assumption_lip_special} hold.
        Then, the minimax solution $\varphi$ of the Cauchy problem \cref{Cauchy_problem} satisfies the following local Lipschitz condition with respect to the functional variable $x(\cdot)$:
        for any compact set $D \subset C([- h, T], \mathbb{R}^n)$, there exists $\lambda_\varphi^\ast \doteq \lambda_\varphi^\ast(D) > 0$ such that
        \begin{equation} \label{varphi_lip_special}
            |\varphi(t, x_1(\cdot)) - \varphi(t, x_2(\cdot))|
            \leq \lambda_\varphi^\ast
            \biggl( \|x_1(t) - x_2(t)\|
            + \sqrt{\int_{- h}^{t} \|x_1(\tau) - x_2(\tau)\|^2 \, \rd \tau} \biggr)
        \end{equation}
        for all $t \in [0, T]$ and $x_1(\cdot)$, $x_2(\cdot) \in D$.
    \end{theorem}

    The proof of \cref{theorem_main_2} follows essentially the same lines as that of \cref{theorem_main} except for the choice of the auxiliary functional $\gamma$.
    The proposed construction of the functional $\gamma$ is close to \cite{Lukoyanov_2010_IMM_Eng_1,Lukoyanov_2010_IMM_Eng_2} and goes back to \cite[Section 34]{Krasovskii_1963}, where stability problems for time-delay systems via Lyapunov--Krasovskii functionals are studied.

    \begin{proof}[Proof of \cref{theorem_main_2}]
        Let a compact set $D \subset C([- h, T], \mathbb{R}^n)$ be fixed.
        Take the set $K \doteq K(D)$ as in step 1 of the proof of \cref{theorem_main} and choose $\lambda_H^\ast \doteq \lambda_H^\ast(K)$ and $\lambda_\sigma^\ast \doteq \lambda_\sigma^\ast(K)$ according to conditions (i) and (ii) from \cref{assumption_lip_special}, respectively.
        Denote $\omega \doteq 4 J \lambda_H^\ast \max \{ 1, 1 / (\sqrt{2} \lambda_\sigma^\ast) \}$ and consider the functional
        \begin{equation} \label{gamma_2}
            \gamma(t, x(\cdot))
            \doteq \|x(t)\|^2
            + \omega \sum_{j = 1}^J \int_{t - \vartheta_j}^{t} \|x(\tau)\|^2 \, \rd \tau
            + \int_{- h}^{t} \|x(\tau)\|^2 \, \rd \tau
        \end{equation}
        for all $(t, x(\cdot)) \in [0, T] \times C([- h, T], \mathbb{R}^n)$.
        One can show that $\gamma$ is $ci$-smooth and
        \begin{equation} \label{gamma_2_derivatives}
            \partial_t \gamma (t, x(\cdot))
            = (\omega J + 1) \|x(t)\|^2 - \omega \sum_{j = 1}^J \|x(t - \vartheta_j)\|^2,
            \quad \nabla \gamma (t, x(\cdot))
            = 2 x(t)
        \end{equation}
        for all $(t, x(\cdot)) \in [0, T) \times C([- h, T], \mathbb{R}^n)$.
        Consider the function
        \begin{equation} \label{C_2}
           a(t)
           \doteq 6 \lambda_H^\ast (e^{(\omega J + 1 + 6 \lambda_H^\ast) (T - t) / 2} - 1) / (\omega J + 1 + 6 \lambda_H^\ast)
           + \sqrt{2} \lambda_\sigma^\ast e^{(\omega J + 1 + 6 \lambda_H^\ast) (T - t) / 2}
        \end{equation}
        for all $t \in \mathbb{R}$.
        Note that $a(t) > 0$ for all $t \in [0, T]$, $a(T) = \sqrt{2} \lambda_\sigma^\ast$, and
        \begin{equation} \label{C_2_differential_equation}
            2 \dot{a}(t) + (\omega J + 1 + 6 \lambda_H^\ast) a(t) + 6 \lambda_H^\ast
            = 0
            \quad \forall t \in \mathbb{R}.
        \end{equation}
        Put $\lambda_\varphi^\ast \doteq \lambda_\varphi^\ast(D) \doteq a(0) \sqrt{\omega J + 1}$.
        Hence, in order to complete the proof, it remains to show that
        \begin{equation} \label{basic_2}
            \varphi(t, x_1(\cdot)) - \varphi(t, x_2(\cdot))
            \leq a(t) \sqrt{\gamma(t, x_1(\cdot) - x_2(\cdot))}
        \end{equation}
        for all $t \in [0, T]$ and $x_1(\cdot)$, $x_2(\cdot) \in D$.
        Indeed, for any $t \in [0, T]$ and $x_1(\cdot)$, $x_2(\cdot) \in D$, it follows from \cref{basic_2} and the inequality $a(t) \leq a(0)$ that
        \begin{displaymath}
            \begin{aligned}
                & \varphi(t, x_1(\cdot)) - \varphi(t, x_2(\cdot))
                \leq a(0) \sqrt{\gamma(t, x_1(\cdot) - x_2(\cdot))} \\
                & \quad \leq a(0) \sqrt{\|x_1(t) - x_2(t)\|^2 + (\omega J + 1) \int_{- h}^{t} \|x_1(\tau) - x_2(\tau)\|^2 \, \rd \tau} \\
                & \quad \leq \lambda_\varphi^\ast \biggl( \|x_1(t) - x_2(t)\|
                + \sqrt{\int_{- h}^{t} \|x_1(\tau) - x_2(\tau)\|^2 \, \rd \tau} \biggr).
            \end{aligned}
        \end{displaymath}

        Fix $\varepsilon > 0$ and, by the functional $\gamma$ from \cref{gamma_2} and the function $a(\cdot)$ from \cref{C_2}, define the functional $\nu$ according to \cref{nu}.
        The functional $\nu$ is $ci$-smooth and
        \begin{displaymath}
            \begin{aligned}
                \partial_t \nu(t, x(\cdot))
                & = \dot{a}(t) \sqrt{\varepsilon + \gamma(t, x(\cdot))}
                + \frac{a(t) \partial_t \gamma (t, x(\cdot))}{2 \sqrt{\varepsilon + \gamma(t, x(\cdot))}}, \\
                \nabla \nu(t, x(\cdot))
                & = \frac{a(t) \nabla \gamma (t, x(\cdot))}{2 \sqrt{\varepsilon + \gamma(t, x(\cdot))}}
            \end{aligned}
        \end{displaymath}
        for all $(t, x(\cdot)) \in [0, T) \times C([- h, T], \mathbb{R}^n)$.
        Note that, in view of \cref{gamma_2,gamma_2_derivatives},
        \begin{equation} \label{nu_2_derivatives_estimates}
            \begin{aligned}
                \partial_t \nu(t, x(\cdot))
                & \leq \biggl( \dot{a}(t) + \frac{(\omega J + 1) a(t)}{2} \biggr)
                \sqrt{\varepsilon + \gamma(t, x(\cdot))} \\
                & \quad - \frac{\omega a(t)}{2 \sqrt{\varepsilon + \gamma(t, x(\cdot))}} \sum_{j = 1}^J \|x(t - \vartheta_j)\|^2, \\
                \|\nabla \nu(t, x(\cdot))\|
                &\leq a(t)
            \end{aligned}
        \end{equation}
        for all $(t, x(\cdot)) \in [0, T) \times C([- h, T], \mathbb{R}^n)$.

        Let $t \in [0, T]$ and $x_1(\cdot)$, $x_2(\cdot) \in D$ be fixed.
        Arguing as in steps 2 and 3 of the proof of \cref{theorem_main}, consider the set $S$ of solutions of the corresponding differential inclusion \cref{F_differential_inclusion} under the initial conditions \cref{F_initial_condition} and conclude that there exists a triple $(y_1^\circ(\cdot), y_2^\circ(\cdot), z^\circ(\cdot)) \in S$ satisfying \cref{z^circ}.

        Define the corresponding function $\mu(\cdot)$ by \cref{mu}.
        Then, based on \cref{dot_mu_estimate} and the choice of $\lambda_H^\ast$, we derive
        \begin{displaymath}
            \begin{aligned}
                \dot{\mu}(\tau)
                & \leq \partial_t \nu(\tau, \Delta y^\circ(\cdot))
                + \lambda_H^\ast (1 + \|\nabla \nu(\tau, \Delta y^\circ(\cdot))\|) \\
                & \quad \times \biggl( \|\Delta y^\circ(\tau)\|  + \sum_{j = 1}^J \|\Delta y^\circ(\tau - \vartheta_j)\|
                + \sqrt{\int_{- h}^{\tau} \|\Delta y^\circ(\xi)\|^2 \, \rd \xi} \biggr)
            \end{aligned}
        \end{displaymath}
        for a.e. $\tau \in [t, T]$, where we denote $\Delta y^\circ(\cdot) \doteq y_1^\circ(\cdot) - y_2^\circ(\cdot)$.
        Multiplying both sides of this inequality by $2 \sqrt{\varepsilon + \gamma(\tau, \Delta y^\circ(\cdot))}$ and then using \cref{nu_2_derivatives_estimates} as well as the inequalities
        \begin{displaymath}
            \begin{aligned}
                \|\Delta y^\circ(\tau)\| + \sqrt{\int_{- h}^{\tau} \|\Delta y^\circ(\xi)\|^2 \, \rd \xi}
                & \leq 2 \sqrt{\varepsilon + \gamma(\tau, \Delta y^\circ(\cdot))}, \\
                \sqrt{\varepsilon + \gamma(\tau, \Delta y^\circ(\cdot))}
                \sum_{j = 1}^J \|\Delta y^\circ(\tau - \vartheta_j)\|
                & \leq \varepsilon + \gamma(\tau, \Delta y^\circ(\cdot))
                + J \sum_{j = 1}^J \|\Delta y^\circ(\tau - \vartheta_j)\|^2,
            \end{aligned}
        \end{displaymath}
        we get
        \begin{displaymath}
            \begin{aligned}
                2 \sqrt{\varepsilon + \gamma(\tau, \Delta y^\circ(\cdot))} \dot{\mu}(\tau)
                & \leq \bigl( 2 \dot{a}(\tau) + \omega (J + 1) a(\tau) + 6 \lambda_H^\ast (1 + a(\tau)) \bigr)
                \bigl( \varepsilon + \gamma(\tau, \Delta y^\circ(\cdot)) \bigr) \\
                & \quad + \bigl( 2 J \lambda_H^\ast (1 + a(\tau)) - \omega a(\tau) \bigr)
                \sum_{j = 1}^J \|\Delta y^\circ(\tau - \vartheta_j)\|^2 \\
            \end{aligned}
        \end{displaymath}
        for a.e. $\tau \in [t, T]$.
        The first term in the right-hand side of this inequality vanishes since $a(\cdot)$ satisfies the differential equation \cref{C_2_differential_equation}.
        In addition, for every $\tau \in [t, T]$, noting that $a(\tau) \geq a(T) = \sqrt{2} \lambda_\sigma^\ast$ and recalling the definition of $\omega$, we obtain
        \begin{displaymath}
            2 J \lambda_H^\ast (1 + a(\tau)) - \omega a(\tau)
            \leq 2 J \lambda_H^\ast - \omega \lambda_\sigma^\ast / 2
            + (2 J \lambda_H^\ast - \omega / 2) a(\tau)
            \leq 0.
        \end{displaymath}
        As a result, we find that $\dot{\mu}(\tau) \leq 0$ for a.e. $\tau \in [t, T]$ and, hence, $\mu(t) \geq \mu(T)$.

        In view of \cref{boundary_condition,z^circ} and thanks to the choice of $\lambda_\sigma^\ast$, we have
        \begin{displaymath}
            \begin{aligned}
                \mu(T)
                & \geq a(T) \sqrt{\varepsilon + \gamma(T, \Delta y^\circ(\cdot))}
                + \sigma(y_2^\circ(\cdot)) - \sigma(y_1^\circ(\cdot)) - \varepsilon (T - t) \\
                & \geq \sqrt{2} \lambda_\sigma^\ast \sqrt{\gamma(T, \Delta y^\circ(\cdot))}
                - \lambda_\sigma^\ast \biggl( \|\Delta y^\circ(T)\|
                + \sqrt{\int_{- h}^{T} \|\Delta y^\circ(\tau)\|^2 \, \rd \tau} \biggr)
                - \varepsilon (T - t) \\
                & \geq - \varepsilon (T - t).
            \end{aligned}
        \end{displaymath}
        Thus, taking \cref{mu_t} into account, we come to inequality \cref{last_inequality}, which yields \cref{basic_2} if we pass to the limit as $\varepsilon \to 0^+$.
    \end{proof}

    As in \cref{section_main}, analyzing the proof of \cref{theorem_main_2}, we derive the corresponding global Lipschitz continuity result for the minimax solution $\varphi$.

    \begin{assumption} \label{assumption_lip_special_glob}
        The following conditions are satisfied.

        \smallskip

        \noindent (i)
            There exist numbers $J \in \mathbb{N}$, $\vartheta_j \in (0, h]$, where $j \in \overline{1, J}$, and $\lambda_H^\ast > 0$ such that inequality \cref{H_lip_special} is valid for all $t \in [0, T]$, $x_1(\cdot)$, $x_2(\cdot) \in C([- h, T], \mathbb{R}^n)$, and $s \in \mathbb{R}^n$.

        \smallskip

        \noindent (ii)
            There exists a number $\lambda_\sigma^\ast > 0$ such that inequality \cref{sigma_lip_special} is valid for all $x_1(\cdot)$, $x_2(\cdot) \in C([- h, T], \mathbb{R}^n)$.
    \end{assumption}

    \begin{corollary} \label{corollary_main_2_glob}
        Let conditions {\rm (i)} and {\rm (ii)} from \cref{assumption_basic} and \cref{assumption_lip_special_glob} hold.
        Then, the minimax solution $\varphi$ of the Cauchy problem \cref{Cauchy_problem} satisfies the following condition: there exists $\lambda_\varphi^\ast > 0$ such that inequality \cref{varphi_lip_special} is valid for all $t \in [0, T]$ and $x_1(\cdot)$, $x_2(\cdot) \in C([- h, T], \mathbb{R}^n)$.
    \end{corollary}

    \begin{remark}
        In connection with \cref{theorem_main_2,corollary_main_2_glob}, an observation analogous to \cref{remark_definitions_Lip} can be made.
        Here, we limit ourselves to mentioning papers \cite{Plaksin_2020_JOTA,Plaksin_2021_SIAM} where Lipschitz conditions close to those from \cref{theorem_main_2,corollary_main_2_glob} are included in definitions of viscosity solutions of Cauchy problems for path-dependent Hamilton--Jacobi equations.
    \end{remark}

    In addition, as another corollary of \cref{theorem_main_2}, we obtain an infinitesimal criterion for the minimax solution $\varphi$ of the Cauchy problem \cref{Cauchy_problem} in terms of a pair of inequalities for lower and upper derivatives of $\varphi$ in finite-dimensional (or single-valued) directions \cite[Section 5.4]{Lukoyanov_2006_IMM_Eng}.
    Note that similar criteria are known for the value functional of a differential game for a time-delay system \cite[Theorem 3]{Lukoyanov_2010_IMM_Eng_2} (see also, e.g., \cite[Section 3.7]{Gomoyunov_Lukoyanov_RMS_2024}), i.e., in the case where the Hamiltonian $H$ is of Isaacs form.

    Recall that, for a functional $\varphi \colon [0, T] \times C([- h, T], \mathbb{R}^n) \to \mathbb{R}$, its {\it lower} and {\it upper} (right) {\it derivatives} at a point $(t, x(\cdot)) \in [0, T) \times C([- h, T], \mathbb{R}^n)$ in a {\it finite-dimensional direction} $f \in \mathbb{R}^n$ are defined by
    \begin{displaymath}
        \begin{aligned}
            \partial_- \{\varphi(t, x(\cdot)) \mid f \}
            & \doteq \liminf_{\delta \to 0^+}
            \frac{\varphi(t + \delta, y^{(f)}(\cdot)) - \varphi(t, x(\cdot))}{\delta}, \\
            \partial_+ \{\varphi(t, x(\cdot)) \mid f \}
            & \doteq \limsup_{\delta \to 0^+}
            \frac{\varphi(t + \delta, y^{(f)}(\cdot)) - \varphi(t, x(\cdot))}{\delta},
        \end{aligned}
    \end{displaymath}
    where $y^{(f)}(\cdot) \in \Lip(t, x(\cdot))$ is given by $y^{(f)}(\tau) \doteq x(t) + f(\tau - t)$ for all $\tau \in (t, T]$.

    \begin{corollary} \label{corollary_criteria}
        Let conditions {\rm (i)} and {\rm (ii)} from \cref{assumption_basic} and \cref{assumption_lip_special} hold.
        Then, a functional $\varphi \colon [0, T] \times C([- h, T], \mathbb{R}^n) \to \mathbb{R}$ is the minimax solution of the Cauchy problem \cref{Cauchy_problem} if and only if $\varphi$ is non-anticipative and continuous, satisfies the boundary condition \cref{boundary_condition}, the Lipschitz condition from \cref{theorem_main_2}, and the inequalities
        \begin{equation} \label{directional_derivatives_retarded_Lip}
            \begin{aligned}
                \min_{f \in B(t, x(\cdot))}
                \bigl( \partial_- \{\varphi(t, x(\cdot)) \mid f\} - \langle s, f \rangle \bigr)
                + H(t, x(\cdot), s)
                & \leq 0, \\
                \max_{f \in B(t, x(\cdot))}
                \bigl( \partial_+ \{\varphi(t, x(\cdot)) \mid f\} - \langle s, f \rangle \bigr)
                + H(t, x(\cdot), s)
                & \geq 0
            \end{aligned}
        \end{equation}
        for all $(t, x(\cdot)) \in [0, T) \times C([- h, T], \mathbb{R}^n)$ and $s \in \mathbb{R}^n$.
    \end{corollary}

    In \cref{directional_derivatives_retarded_Lip}, $B(t, x(\cdot))$ denotes the ball in $\mathbb{R}^n$ with centre at the origin and radius $c_H (1 + \|x(\cdot \wedge t)\|_\infty)$, the number $c_H$ is taken from \cref{assumption_basic}, (ii).
    Note that the minimum and the maximum in \cref{directional_derivatives_retarded_Lip} are attained since, due to the Lipschitz condition from \cref{theorem_main_2}, the functions $f \mapsto \partial_- \{\varphi(t, x(\cdot)) \mid f\}$ and $f \mapsto \partial_+ \{\varphi(t, x(\cdot)) \mid f\}$ are lower and upper semicontinuous, respectively.

    To verify \cref{corollary_criteria}, it is sufficient to apply the infinitesimal criterion for the minimax solution of the Cauchy problem \cref{Cauchy_problem} in terms of lower and upper derivatives in multi-valued directions \cite[Theorem 5]{Gomoyunov_Lukoyanov_RMS_2024} (see also \cite[Theorem 8.1]{Lukoyanov_2003_1}) and then use the formulas for calculating the derivatives in multi-valued directions via the derivatives in finite-dimensional directions in the case of functionals satisfying the Lipschitz condition from \cref{theorem_main_2} \cite[Lemma 1]{Lukoyanov_2010_IMM_Eng_2}.

    We conclude the paper with the observation that the method for proving Lipschitz continuity properties of the minimax solution $\varphi$ proposed in the present paper seems to be quite universal and developable to other assumptions regarding Lipschitz continuity properties of the Hamiltonian $H$ and/or the boundary functional $\sigma$.

\end{document}